\documentclass[11pt]{amsart}
\usepackage[utf8]{inputenc}

\usepackage{amssymb}
\usepackage{graphics}
\usepackage{graphicx}
\usepackage{amscd}
\usepackage{color}
\usepackage{amsmath, amsthm, epsfig,  psfrag}
\usepackage[dvipdfm]{hyperref}
\hypersetup{
    colorlinks,
    citecolor=black,
    filecolor=black,
    linkcolor=black,
    urlcolor=black
}
\usepackage{pinlabel}

\newtheorem{theorem}{Theorem}[section]
\newtheorem{lemma}[theorem]{Lemma}
\newtheorem{prop}[theorem]{Proposition}
\newtheorem{cor}[theorem]{Corollary}

\theoremstyle{definition}
\newtheorem{definition}[theorem]{Definition}

\renewcommand{\d}{\partial}

\newcommand{\h}{\widehat}
\newcommand{\HFK}{\h{HFK}}
\newcommand{\CFK}{\h{CFK}}

\newcommand{\D}{\mathcal{D}}

\newcommand{\hthe}{\hat{\theta}}

\newcommand{\CC}{\mathbb{C}}

\newcommand{\RR}{\mathbb{R}}
\newcommand{\ZZ}{\mathbb{Z}}
\newcommand{\HH}{\mathbb{H}}

\newcommand{\ti}{\tilde}

\author{Olga Plamenevskaya}
\address{Department of Mathematics, SUNY Stony
  Brook, Stony Brook, NY 11794}
\email{olga@math.stonybrook.edu}

\title{Transverse invariants and right-veering}

\begin{document}

\begin{abstract} 
 A closed braid $\beta$ naturally gives rise to a transverse link $L$ in the standard contact 3-space. We study the effect of the dynamical properties of the monodromy of $\beta$, such as right-veering, on the contact-topological properties of $L$  and the values of transverse invariants in 
 Heegaard Floer and Khovanov homologies. In particular, for a 3-braid $\beta$, we show that  $\hthe(L) \in \HFK(m(L))$ is non-zero if and only if $\beta$ is right-veering. 
For higher index braids,  $\hthe(L)\neq 0$  whenever $\beta$ has  the fractional
Dehn twist coefficient $C>1$. We use grid diagrams and the structure of
Dehornoy's braid ordering.
\end{abstract}

\maketitle

\vspace{-.3cm}

\section{Introduction}

Links that are transverse to the contact planes in the standard contact 3-space 
$(\RR^3, \xi_{std})= \ker (dz+ r^2 d\varphi)$ can be
described via braids around the $z$-axis (up to a stabilization
equivalence). In this paper, we would like to see how  
dynamics of the braid monodromy affects the contact-topological properties of
the corresponding transverse link. We focus on transverse link invariants
that live in Heegaard Floer and Khovanov homologies of the link $K$: the
transverse invariant $\hthe(K) \in \HFK (m(K))$ \cite{OST} as well as the
invariant $\psi(K) \in Kh(K)$ \cite{Pla-Kh}. A lot of previous research was
aimed at using these invariants to detect {\em transverse non-simplicity}, i.e.
to show that certain transverse knots are not isotopic despite having the same
smooth type and self-linking number. The Heegaard Floer homological invariant
$\hthe$ detects many such  pairs, but it is  unknown whether the
Khovanov homological invariant $\psi$ is effective.

The purpose of this paper is to clarify the
qualitative  meaning of transverse
invariants of an {\em individual} transverse knot. It turns out  that the
behaviour of $\hthe\in \HFK$ is related to certain dynamical properties of the
braid monodromy (acting on a disk $\D$ with punctures). A braid is {\em right-veering} if its  
monodromy sends {\em every} arc connecting $\d \D$ to one of the punctures
to the right of itself; see Section \ref{defns} for a
detailed definition. (Otherwise, we say that the braid is non-right-veering.)  
The following proposition is known and has motivated the
present paper.

\begin{prop} (cf. \cite{BG,Pla-Kh,BVV}) \label{vanish} Suppose $K$ is a
transverse link with a
non-right-veering braid representative. Then $\hthe(K)=0$ and $\psi(K) =0$.  
\end{prop}
We establish a partial converse to this fact:
\begin{theorem} \label{3b} Suppose that $K$ is a transverse 3-braid. Then $K$ is
right-veering if and only if $\hthe(K)\neq 0$.
\end{theorem}

\begin{theorem} \label{C>1} Suppose $K$ is a transverse braid with the fractional
Dehn twist coefficient $C>1$. Then $\hthe(K) \neq 0$. 
\end{theorem}

The fractional Dehn twist coefficient measures the amount of
positive rotation the braid monodromy makes around the boundary of the disk. The
condition $C>1$  means, roughly, that the braid has more positive rotation than
a full positive twist on all strands. We will give the precise definition of
FDTC in Section \ref{FDTC}.

It is well-known that right-veering property of diffeomorphisms of surfaces
plays an important role in 3-dimensional contact topology. Namely, by Giroux's
theorem contact structures can be described via compatible open books, 
and tight contact structures are exactly those for which  {every} compatible
open book is right-veering \cite{HKM}.  A product of positive Dehn twists
is a special case of a right-veering diffeomorphism; an open book with such
monodromy supports a Stein fillable contact structure.  Some of the properties
of the monodromy of an open book are reflected, via the compatible contact
structure,  by the Ozsv\'ath--Szab\'o contact invariant \cite{OS-cont}. More
precisely, this invariant
vanishes for open books with non-right-veering monodromy (i.e. for overtwisted
contact structures), and is always non-zero for monodromies which are products
of positive Dehn twists (i.e. for Stein fillable contact structures). 
The fractional Dehn twist coefficient of the open book
monodromy also carries information on the contact structure. By \cite{HKM2}, 
a contact structure supported by an open book with connected boundary and
pseudo-Anosov monodromy with $C\geq 1$ is isotopic to a perturbation of a taut
foliation, and therefore weakly symplectically fillable. Then by
\cite{OS-genus}, the Heegaard Floer contact invariant (with {\em twisted}
coefficients) is non-zero. 

Our Theorem \ref{C>1} for classical braids is  parallel to the result in the
contact structures context, but our proof is surprisingly different. Taut
foliations and fillability seem irrelevant for braids; instead, we use braid
structure supplied by Dehornoy's theory of braid orderings. This allows to 
prove  non-vanishing of $\hthe$ via grid diagrams. Note that a generalization 
of $\hthe$ for the case of braids in arbitrary open books was given in
\cite{BVV}; the definition is reminiscent of that of the Heegaard Floer contact 
invariant. It would be interesting to find a technique that could be used to
prove both a general version of 
Theorem \ref{C>1} and the result of of \cite{HKM2}.  
An additional subtlety is that our Theorem~\ref{C>1} holds with {\em
untwisted} coefficients (we use $\ZZ/2$ coefficients throughout); in the context
of open books, non-vanishing result with $\ZZ/2$ coefficients has not been
established. (In general, the untwisted version of the Ozsv\'ath-Szab\'o
invariant may vanish for weakly fillable contact structures \cite{Ghi1}.)

Invariants related to braid orderings encode some of the topology of
the braid closure, such as
genus bounds for the underlying knot, existence of certain isotopies, and 
some information on the geometry of the link complement \cite{It1, It2,
MN}. In this paper,  we show that braid orderings are useful in contact
topology. It would be interesting to develop further connections and
applications. 

Right-veering seems important for
certain contact-topological properties of
transverse links,  so we take a closer look at features of right-veering
transverse braids. An important remark is in order: 
it is possible to start with a non-right-veering braid and stabilize it, while 
preserving the transverse link type, to obtain right-veering monodromy
(Proposition \ref{prop-stab-rv}).
We therefore suggest

\begin{definition} Let $K$ be a transverse link in $(S^3, \xi_{std})$. We say
that $K$ is right-veering if {\em every} braid representative of $K$ is a
right-veering braid. Similarly, $K$ is non-right-veering if {\em there exists}
a non-right-veering braid representative for $K$.
\end{definition}

It is interesting to ask if $\hthe(K)$ is always non-zero whenever $K$ is a
right-veering transverse link; at the moment, we do not have any
counterexamples. Another question concerns the behavior of the
Khovanov-homological invariant $\psi(K)$ of a right-veering link $K$. 
The analogs of Theorems \ref{3b} and \ref{C>1} do not hold true for $\psi(K)$,
and there are examples of right-veering transverse links with vanishing $\psi$.
We include a few observations concerning $\psi$ in this paper.

In general, we would like to claim that right-veering transverse links have
certain special
properties. It turns out that right-veering is a
necessary condition for tightness of 
the double cover $(\Sigma(K), \xi_K)$ of $(S^3, \xi_{std})$ branched over $K$,
and for equality in various versions of the
slice-Bennequin inequality. Recall that $sl(K)$ has upper bounds \cite{Ru,
Pla-tb, Pla-Kh, Shu} in terms of the slice 
genus $g^*(K)$ of the knot $K$, the Ozsv\'ath-Szab\'o 
concordance invariant $\tau(K)$, and Rasmussen's concordance invariant $s(K)$:
$$
sl(K) \leq 2 g^*(K)-1, \qquad sl(K) \leq 2 \tau(K) -1, \qquad sl(K) \leq s(K)
-1.
$$
We show that these upper bounds
can only be achieved by right-veering knots.

\begin{theorem} \label{sl-bounds} Suppose that $K$ is a non-right-veering
transverse knot.
Then 

(1) $sl(K) <s -1$, $sl(K)< 2\tau -1$, and $sl(K) < 2 g^*(K)-1$.

(2) The branched double cover $(\Sigma(K), \xi_K)$   is overtwisted.

\end{theorem}

These facts are straightforward observations or corollaries of known results; 
we collect them together to create a coherent picture. The converse statements
to those of
Theorem \ref{sl-bounds} are not true (see Proposition \ref{counterex}),
although 
it is well-known that when $K$ is a {\em quasipositive} braid, the above bounds
for the self-linking number become equalities, and the branched double cover
is tight (even Stein fillable). We discuss this further in Section
\ref{section-nonrv}.

It might be tempting to think of the right-veering vs. non-right-veering
transverse knots as a dichotomy analogous to that of tight vs. overtwisted
contact structures, but this analogy does not go very far. Indeed, the main
feature of overtwisted contact structures is that they satisfy an $h$-principle
\cite{El-ot}, in particular, two overtwisted contact structures are isotopic
whenever their underlying plane fields are homotopic. In the case of transverse
knots, an $h$-principle would say that non-right-veering knots are transversely
simple
(i.e. the transverse isotopy type would be determined by the smooth knot and
the self-linking number), but this is  not true: some of the transversely
non-simple braids given in \cite{BM} are not right-veering.

The paper is organized as follows. In Section 2, we review the necessary
definitions and background. In Section 3, we prove Theorem \ref{sl-bounds}
along with some other easy properties of non-right-veering braids. 
In Section 4, we discuss 3-braids and prove Theorem
\ref{3b}. Section 5 contains the proof of Theorem \ref{C>1} and the related
discussion of Dehornoy's floor function and the fractional Dehn twist
coefficient.

\vspace{2mm}

 {\bf Acknowledgements.} The author is grateful to John Baldwin, John Etnyre,
Marco Golla, and Jeremy Van Horn-Morris for helpful conversations.

\section{Background} \label{defns}

\subsection{Braids and transverse links} \label{braids} As usual, we will write
the elements of the Artin braid group $B_m$ (braids on $m$ strands) as braid
words on standard generators, $\sigma_1$, $\sigma_2$, \dots, $\sigma_{m-1}$. We
draw braids from left to right, and consider braid closures as closed  braids
around the $z$-axis in $\RR^3$. (Closed braids correspond to conjugacy classes
in $B_m$.) 

If $\RR^3$ is endowed with the standard contact structure 
$\xi_{std}=\ker(dz + r^2 \, d \phi)$, closed braids around the $z$-axis
naturally give rise to transverse links. Moreover, two braids that produce 
transversely isotopic links are related by braid isotopies and positive Markov
stabilization and destabilization moves \cite{OrSh} (this fact is usually
called
``transverse Markov theorem''). Positive Markov stabilization of a braid 
$\beta \in B_m$ gives the braid $\beta \sigma_m \in B_{m+1}$; geometrically, 
this means adding an extra strand and a ``positive kink'' to the given closed
braid. Negative Markov stabilization  of $\beta$ gives the braid $\beta
\sigma_m^{-1}$ which is {\em not} transversely isotopic to $\beta$. When we
consider the transverse links corresponding to braids,  negative
Markov stabilization is called  transverse stabilization. A
basic invariant of transverse links, self-linking number,
can be computed from a braid via the formula 
$$
sl = \# \{\text{positive
crossings} \} - \#\{\text{negative crossings} \} - (\text{braid index}).
$$ 
Transverse stabilization decreases the self-linking number by 2.

The braid group $B_m$ is naturally identified with a the mapping class group of
a disk with $m$ marked points.  (We will assume that $\D \subset \CC$ is the
standard unit disk, and that the set  $Q = \{x_1, x_2, \dots x_m\}$ of the
marked points lies on the $x$-axis so that $x_1 <x_2<\dots <x_m$.) The braid
group acts on $(\D, Q)$ on the right (fixing $\d \D$).  Each standard generator
$\sigma_i$ acts by a right-handed half-twist interchanging $x_i$ and $x_{i+1}$,
supported in a small neighborhood of the segment $[x_i, x_{i+1}]$.

\subsection{Right-veering property and braid orderings} We now recall the
definition of right-veering braids. For this, we consider the action of the
braid monodromy on arcs that connect 
the boundary $\d \D$ of the disk to one of the marked points in $Q$, while
avoiding all other marked points. We will refer to such an arc as ``an arc in
$(\D, Q)$''. Suppose that $a_1$ and $a_2$ are non-isotopic arcs in $(\D, Q)$ 
that start
at the same point in $\d \D$ and intersect transversely without non-essential
intersections (that is, the two arcs form no bigons in $\D \setminus Q$, so that
no interesections can be removed by an isotopy of the arcs). We will say that
$a_1$ lies to the right of $a_2$ if the pair of tangent vectors 
$(\dot{a}_1, \dot{a}_2)$ at the arcs' initial point induces the original
orientation on the disk $\D$.
For arcs with non-essential intersections, the same definition can be used
after the arcs are {\em pulled taut}, i.e. isotoped to eliminate all
non-essential intersections. Now, let $\beta \in B_m$ be a braid, and write 
$\beta(a)$ for the image of an arc $a$ in $(\D, Q)$ under the action of $\beta$.
(The braid group acts on the right but we abuse notation.)

\begin{definition} Let  $\beta \in B_m$ be a braid.
We say that $\beta$ is right-veering if for  any
arc
$a$ in $(\D, Q)$,  the
arc $\beta(a)$ is either isotopic to $a$ or lies to the right of $a$ (after
the arcs $a$ and $\sigma(a)$ are pulled taut in $\D \setminus Q$).

If there is an arc $a$ such that $\beta(a)$ is to the left of $a$ (i.e. $a$ is
to the right of $\beta(a)$), we say that $\beta$ is non-right-veering.
\end{definition}

In the context of open book decompositions, the notion of right-veering was
developed in \cite{HKM} and applied to the study of contact structures. For
braids, a definition where one considers arcs with both endpoints in $\d \D$
was given in \cite{BVV};  the present definition appears in \cite{BG}. The
general idea of right-veering goes back to W. Thurston.

Using Thurston's approach, the idea of right-veering can be used to 
show that the braid group $B_m$ is orderable. (Namely, there exists a {\em
left-invariant} linear order on $B_m$, so that if $\beta_1 < \beta_2$, then 
 $\gamma \beta_1 < \gamma \beta_2$  for any $\gamma \in B_m$.) Many different
orderings (known as {\em Nielsen--Thurston} orderings) arise from this idea; 
we refer the reader to \cite{De} for a general discussion. In this paper, we
will only need a specific ordering introduced by Dehornoy \cite{De1}.
Historically, Dehornoy was the first one to describe a linear
order on $B_m$; he found it from an algebraic perspective. 
 By definition,
$\beta  \succ 1$ in Dehornoy's order iff the braid $\beta$ admits a
braid word that contains the generator $\sigma_i$ but no
$\sigma_i^{-1}$ and no $\sigma_j$ for $j<i$. (Such words are called
$\sigma$-positive.) For example, we have 
$\sigma_3 \sigma_1^2 \sigma_2^{-5} \sigma_3^{-1} \succ 1$ and $\sigma_2 \sigma_5
\sigma_3^{-2} \sigma_2^2 \succ 1$. Now, we set  $\beta \succ
\beta'$ if $(\beta')^{-1}\beta  \succ 1$, i.e. $(\beta')^{-1}\beta$ admits a
$\sigma$-positive word.  This gives a well-defined
left-invariant order, \cite{De1, De}. (From this perspecitive, it is
non-trivial 
to check, for example, that $\beta$ and $\beta^{-1}$ can never be both 
$\sigma$-positive.)

We will also use some more specific terminology. We say that a braid word 
is $\sigma_i$-positive if it contains at least one letter  $\sigma_i$ but no
$\sigma_i^{-1}$ and no $\sigma_j^{\pm 1}$ with $j<i$.  Similarly, we can define
$\sigma_i$-negative words. We will say that a word is $\sigma_i$-free if 
it contains no $\sigma_j^{\pm 1}$ for $j \leq i$.

Dehornoy's order can be given a geometric interpretation \cite{FGRRW}, by
considering the action of the braid monodromy $\beta \in Map (D, Q)$ on the
$x$-axis. 
(Recall that we assume that the marked points are all on the $x$-axis, 
labeled from left to right.) It turns out that $\beta \succ 1$ in Dehornoy's
order iff the image of the $x$-axis under $\beta$ veers to the right of 
the $x$-axis, when compared at the left endpoint and pulled taut. (Note that the
$x$-axis passes
through marked points, and may coincide with its image on a few initial 
intervals. The image must veer right from the $x$-axis when they eventually
diverge. Indeed, for a $\sigma_1$-positive word, the image of the $x$-axis 
will veer to the right at $\d \D$, but for a word that is $\sigma_2$-free, 
the image will follow the $x$-axis up to the point $x_2$. For a  
$\sigma_3$-positive word, the image of the $x$-axis will veer to the right after
$x_2$.)

It is important to note that the condition 
$\beta \succ 1$ does not imply that $\beta$ is right-veering  (since we do not 
consider images of arbitrary arcs). 
However, if
$\beta$ is non-right-veering, i.e. sends some arc to the left of itself, it has
a
conjugate $\tilde{\beta}$ that sends the initial (leftmost) arc of the $x$-axis
to the left of itself, so $\tilde{\beta} \prec 1$:  

\begin{prop}  \label{sigma-negative} \cite{BG, FGRRW} Any
non-right-veering braid
is conjugate to a braid with a braid word that contains  $\sigma_1^{-1}$
but no $\sigma_1$.  
\end{prop}

\begin{cor} \label{non-rv-form} A non-right-veering transverse link can be
represented by a braid
wuth a word that  contains  $\sigma_1^{-1}$
but no $\sigma_1$. 
\end{cor}

The proposition above and its corollary will be very useful for us.

\subsection{Transverse invariants} In this paper, we will be mostly concerned
with transverse link invariants in Heegaard Floer homology (although we will
discuss Khovanov-homological invariants as well).  
We now very briefly recall the constructions, referring the reader to \cite{OST,
NOT, KN} for details.

A  grid diagram is an $n \times n$ square grid, marked with X's and O's so
that there is exactly one X and exactly one O in each row and each column. 
A grid diagram gives rise to a braid, as follows. First, consider columns of
the grid. For the columns that have O above X, draw a
vertical segment from X to O. For the columns that have X above O, draw two
vertical segments in this column, one from the bottom of the grid to O, the
other from X to the top of the grid. Now, consider the rows, and draw a
horizontal segment in each column connecting X to O, so that horizontal
segments cross {\em over} the vertical segments. The result is an oriented
braid (running from bottom to top); as in subsection \ref{braids}, 
the braid gives a transverse link $K$. See Figure \ref{ex-grid} for an
illustration.

\begin{figure}[ht]
  \includegraphics[scale=0.9, keepaspectratio]{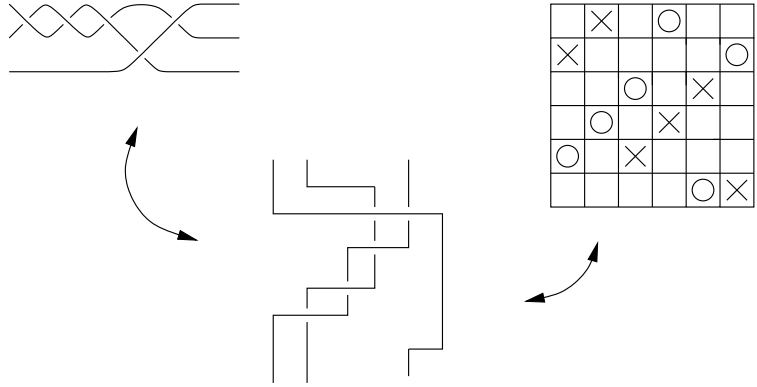}
  \caption{Braids and grid diagrams. The middle picture shows a braid running from bottom to top, constructed from the grid diagram on the right. 
  The picture on the left shows the same braid running from left to right; the braid closure is braided around the (vertical) $z$-axis and represents 
  a transverse link in $(\RR^3, ker(dz +r^2 d \phi))$. }
  \label{ex-grid}
\end{figure}

The grid diagram gives rise to the grid chain complex $\CFK(m(K))$ whose
generators are $n$-tuples of intersection 
points of the grid (the opposite sides of the grid are identified to form a
torus). Here, $m(K)$ is the
mirror knot (the mirroring is needed because one associates  a smooth link to a 
grid diagram in a different way when defining knot Floer homology).
The $n$-tuple given by
the upper right corners of the X markings gives a distinguished cycle in the
complex  $\CFK(m(K))$; the corresponding element  $\hthe \in \HFK(m(K))$
is an invariant of the original transverse knot $K$.

The definition we just gave will be most convenient for our purposes; 
it comes from Khandhawit-Ng braid interpretation
\cite{KN} of the original definition from \cite{OST}. 
The original definition \cite{OST} of the transverse invariant $\hthe(K)$ in
Heegaard Floer
homology starts by  representing the transverse link $K$ in $S^3$ as a push-off
of
a Legendrian link; the latter can be put on a grid, and the corresponding
grid diagram is used to define the transverse invariant (as the class of the
cycle given the upper-right corners of the X's). Note that there are other
ways to define the same invariant. A more topological definition, for
transverse knots in arbitrary contact 3-manifolds, was developed via open books
adapted to the corresponding Legendrian knot in \cite{LOSS}. Finally, yet
another construction, whose input is directly a transverse braid (rather than an
associated Legendrian link), was given in \cite{BVV}, where equivalence of all
the above definitions for transverse knots in $S^3$ was established.
While we work with grid diagrams in this paper, it would
be interesting to reprove and generalize our results from the perspective of
\cite{BVV}. Finally, we remark that there is a related transverse invariant
$\theta^{-}$ in $HFK^-(m(K))$ which potentially contains more information;
however, the hat version will be sufficient for our paper.

The transverse invariant in Khovanov homology, $\psi(K) \in Kh(K)$  was
introduced in \cite{Pla-Kh}. 
For a braid representative of a transverse link $K$, one considers the oriented
braid resolution, and then takes the cycle which is the lowest quantum grading
element in the component of the Khovanov chain complex $CKh(K)$ corresponding
to this resolution. The resulting homology class, $\psi(K)$, is independent of
the braid representative. 

We recall a few properties. 

\begin{prop} \cite{OST, Pla-Kh} Suppose $K_{stab}$ is a transverse stabilization
of
another transverse link, $K$. Then $\hthe(K_{stab})=0$ and $\psi(K_{stab})=0$.
 \label{stabil}
\end{prop}

Both the Heegaard Floer and the Khovanov-homological invariants are functorial
with
respect to a positive crossing resolution in the transverse braid. More
precisely:

\begin{prop} \cite{Baldwin, Pla-Kh} \label{functo} Suppose that the
braid  $\beta$ is obtained from $\beta_{+}$ by removing a positive crossing of
the braid (i.e. removing a generator $\sigma_i$
from the braid word), and let $K$, $K_{+}$ be the corresponding transverse
links. The crossing resolution cobordism gives rise to  homomorphisms on
Khovanov and Heegaard Floer homology, so that 

(1) There exists a map $F: \HFK(m(K_{+})) \to \HFK(m(K))$ such that
$F(\hthe(K_{+}))= \hthe(K)$, 

(2) There exists a map $G:Kh(K_{+}) \to Kh(K)$ such that $G(\psi(K_+))=
\psi(K)$.

Similarly, if $K_{-}$ corresponds to a transverse braid that differs from
$\beta$ by an additional negative crossing, the  maps on Heegaard Floer resp.
Khovanov homology send $\hthe(K)$ to $\hthe(K_{-})$ and $\psi(K)$ to
$\psi(K_{-})$.
\end{prop}

Propositions \ref{sigma-negative}, \ref{stabil}, and \ref{functo} combine to
immediately give a proof of Proposition~\ref{vanish}: non-right-veering
transverse knots have $\hthe=0$ and $\psi=0$. (Note that the proof of
vanishing of $\hthe$ in \cite{BVV} relies on a different definition of
right-veering property, namely, one considers the effect of the braid monodromy
on arcs with both endpoints on $\d \D$. We can, however, reduce our
Proposition \ref{vanish} to \cite{BVV} by ``doubling'' a non-right-veering arc
in $(\D, Q)$.)

\section{Non-right-veering transverse links} \label{section-nonrv}

In this section, we discuss a few basic features of non-right-veering
transverse links. (These  can be turned around to say that right-veering is a
{\em necessary} condition for certain properties.) Recall that those were
defined as links that can be represented by non-right-veering braids. We first
show that non-right-veering braids can become right-veering after positive
Markov stabilizations. This parallels the well-known situation with open books
\cite{HKM}. 

\begin{prop} \label{prop-stab-rv} Let $K$ be an arbitrary  transverse link. Then
$K$ can be always
represented by a right-veering braid.
\end{prop}
 
 \begin{proof} We will use the notion of stabilization along an arc. Recall 
 that the usual Markov stabilization can be thought of as follows. For a braid
$\beta \in B_m = Map (\D, Q)$, we enlarge the disk $\D$ and add an extra point, 
$x_{m+1}$, to the set $Q = \{x_1, x_2, \dots, x_m\}$. The monodromy of the
stabilized braid $\beta \sigma_m$ is then given by the composition of the
monodromy $\beta$ with a positive half-twist interchanging $x_{m}$ and $x_{m+1}$
and supported in a thin neighborhood of a standard arc connecting $x_{m}$ and
$x_{m+1}$. See Figure \ref{stabilize} (left). More generally,  we can instead
consider an {\em arbitrary } arc connecting $x_{m+1}$ to a point  $x_j \in Q$
(and
avoiding $x_i \in Q$ for $i \neq j$), Figure \ref{stabilize} (right). Composing
$\beta$ with a positive half-twist interchanging $x_j$ and $x_{m+1}$ and
supported
in a thin neighborhood of a this arc, we obtain a new stabilized braid. 
Clearly, stabilization along a general arc amounts to the effect of a standard
stabilization composed with conjugation, so the result represents the same
transverse link.  

  \begin{figure}[ht]
  \includegraphics[scale=0.8, keepaspectratio]{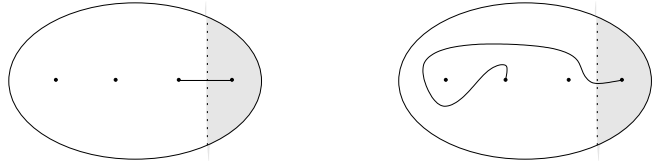}
  \caption{Markov stabilization $\beta \mapsto \beta \sigma_{m}$ (left);
\qquad stabilization along an arbitary  arc (right).}
  \label{stabilize}
\end{figure}
  
Now, we start with an arbitary (non-right-veering) braid $\beta \in B_m$, and
stabilize it twice as described below, so that it becomes right-veering. First,
find a positive integer $N$ so that $\beta$ composed with an $N$-fold positive
boundary twist is right-veering. Enlarge the disk $\D$ so that it is contained
inside a bigger disk $\tilde{\D}$, and $\tilde{\D} \setminus\D$ contains two
additional marked points, $x_{m+1}$ and $x_{m+2}$. We can assume that the
monodromy $\beta$ is supported inside $\D$, and $\beta$ is stabilized along two
arcs that connect $x_{m+1}$ resp. $x_{m+2}$ to some marked points in $\D$. We
choose these two arcs as shown in Figure \ref{stab-arcs}: both arcs make at
least  $N$ full twists in $\tilde{\D} \setminus \D$ before entering $\d \D$.
Moreover, one arc twists in the clockwise direction, and the other in
counterclockwise, and they intersect so that each of them makes at least $N$
twists around $\d \D$ between their first intersection point and the endpoint
 inside $\D$. 
 
   \begin{figure}[ht]
  \includegraphics[scale=0.8, keepaspectratio]{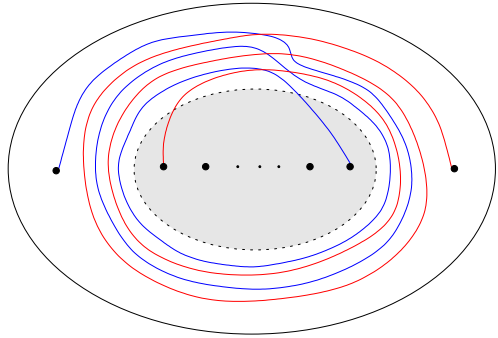}
  \caption{The monodromy of the braid $\beta$ is supported inside the shaded
disk. The two stabilizing arcs twist around.}
  \label{stab-arcs}
\end{figure}
 
 We claim that after positively stabilizing along these two arcs
(in any order), $\beta$ becomes right-veering. Indeed, let $a$ any arc in
$\tilde{\D}$ that starts on $\d \tilde{\D}$ and goes to one of the marked
points. If $a$ does not enter $\D$, then clearly the image of $a$ must lie to
the right of $a$. If $a$ enters $\D$, then it must cross the stabilizing arcs. 
The monodromy of the stabilized braid is the composition of the
monodromy of $\beta$ with two positive half-twists supported near the
stabilizing arcs. Two cases are then possible for the image of $a$ under this
new monodromy: either 1) the image of $a$ veers right toward and around the
point $x_{m+1}$ or $x_{m+2}$, or 2) it follows one of the stabilizing arcs,
making (at least) $N$ right-handed full twists before entering $\D$. In the
first case, the  image
of $a$  must veer  to the right and go around  $x_{m+1}$ or $x_{m+2}$ even
after it is pulled taut, which guarantees that it stays to the right of $a$.
In the second case, the image of $a$ will be to the right of $a$ because
$\beta$ becomes right-veering when composed with $N$ positive full twists. 
 \end{proof}

 \begin{prop} \label{sl} Suppose $K$ is a transverse knot such that $sl(K)= s-1$
or
$sl(K)=2 \tau -1$. Then $K$ is right-veering.
\end{prop}
 
 \begin{proof}  Suppose that $sl(K)=2 \tau(K) -1$. Then right-veering of $K$
follows from Proposition \ref{vanish} and the fact that the equality  $sl(K)=2
\tau(K) -1$ implies that $\hthe(K) \neq 0$. This last fact seems to be absent
from
the literature but is known to experts; indeed, the proof given in
\cite[Proposition 6]{NOT} for the case of thin knots does not actually use the
thin property.  
The idea is that
$\hthe(K)$ 
is related to another invariant, $\theta^-\in HFK^{-}(m(K))$. The latter is 
always non-trivial, and $\hthe$ vanishes if and only if $\theta^-$ is in the
image of the $U$-map. (Recall that one takes the quotient by the image of $U$
to pass from $HFK^{-}$ to $\h{HFK}$.) If $sl(K)= 2 \tau -1$, the grading
level of $\theta^-$ is the same as the very top of the
$\ZZ[U]$-chain in $HFK^-$, which means $\theta^-$ can't be in the image of $U$.

The second statement of the Proposition concerns Rasmussen's invariant $s(K)$,
defined via Khovanov homology. In \cite{Ra} $s$ is only defined for
knots, not for links, but an extension to the case of oriented links was given
in \cite{BW}. Recall that $sl(K) \leq s(K)-1$ when $K$ is a knot; we claim that
the same inequality holds for links. To see this, we will connect the given
link to a knot by a cobordism. Under cobordisms, $s$ changes as follows
\cite{Ra, BW}: if $S$ is a smooth oriented cobordism from $L_1$ to $L_2$ such
that every connected component of $S$ has  boundary in $L_1$, then 
\begin{equation} \label{chi-ineq}
|s(L_2) - s(L_1)| \leq - \chi (S). 
\end{equation}
Given a $(k+1)$-component transverse braid $L$, we can insert $k$ positive
crossings to obtain a connected braid $K$. This gives a cobordism $S$ from the
link $L$ to the knot $K$; we have $\chi(S) = k$. Then $sl(L) = sl(K)-k \leq
s(K)-1-k \leq s(L) -1$ by (\ref{chi-ineq}).  

We are now ready to prove the proposition. Assume that $L$ is a
non-right-veering transverse link. Then by Corollary \ref{non-rv-form}, $L$ has 
a braid representative with only negative crossings in the $\sigma_1$-level 
(i.e. the braid word has entries $\sigma_1^{-1}$ but no $\sigma_1$). Removing
all of these crossings except one, we get a cobordism $S$ from the link $L$ to a
link $L'$ which is a negatively stabilized transverse braid. If $k$ negative
crossings were deleted, then $sl(L) = sl(L') - k$. On the other hand, 
$s(L) \geq s(L') - k$ by  (\ref{chi-ineq}). Since $L'$ is a stabilized
transverse link, the inequality $sl(L') < s(L') -1$ is strict; therefore, 
$sl(L) < s(L) -1$.
\end{proof}

 \begin{prop} \label{branch} Suppose $L$ is a   transverse
link such that the branched double cover  $(\Sigma(L), \xi_L)$  is tight. 
Then $L$ is right-veering.
\end{prop}

\begin{proof} Suppose $L$ is non-right-veering, i.e. there is a
non-right-veering braid representative $\sigma$ for $L$. Then, there is 
an arc $a$ in $(D, Q)$ whose image $\sigma (\alpha)$ can be isotoped to an arc
$b$
minimally intersecting $a$, so that $b$ is to the left of $a$. We want to show
that $(\Sigma(L), \xi_L)$ is overtwisted.

The branched double
cover has 
an open book decomposition $(S, \phi)$ such that $S$ is a double cover of $D$
with branch set $Q$,
and the monodromy $\phi$ covers $\sigma$. The arc $a$ lifts to an arc  $\alpha$
in $S$, 
and it is clear that $\phi(\alpha)$ is isotopic to the arc $\beta$ that covers
$b$ and lies to the left of $\alpha$. We can conclude that $\phi$ is
non-right-veering if we check that $\alpha$ and $\beta$ form no
bigons. Indeed, if there were a bigon $B$ formed by $\alpha$ and $\beta$, 
then $B$ can't cover a bigon in $(D, Q)$, and so $B$ must contain a branch
point (unique for the Euler characheristic reasons). But then in a
neighborhood of $B$ the covering  can be modelled on $z \to z^2$, and
this neighborhood admits a covering involution (modelled on $z \to -z$).
The arcs $\alpha$ and $\beta$ are both invariant under this involution,  
so they have to pass through the branch point, and $B$ is in fact the union of
two bigons, each of which covers a bigon downstairs. This is a contradiction
since we assumed that bigons between $a$ and $b$  in $(\D, Q)$ no
longer exist.
\end{proof}

It is natural to ask whether the converses of Propositions \ref{vanish},
\ref{sl}, and \ref{branch} hold true. We have counterexamples to all statements
except the one concerning~$\hthe$.

\begin{prop} \label{counterex} There exist right-veering transverse knots
whose self-linking number is not maximal in the corresponding smooth knot type,
branched double cover is overtwisted, and $\psi=0$.
\end{prop}

\begin{proof}
Consider the transverse knots $K_1$, $K_2$ that are the $(3,2)$-cables of the
trefoil  discovered by Etnyre-Honda \cite{EH}, see also \cite[Proposition
5]{NOT}. These examples show that this smooth knot type is not transversely
simple: the knots $K_1$ and $K_2$ 
have $sl(K_1)=sl(K_2)=3$ but are not transversely
isotopic. Indeed, $K_1$ can be transversely destabilized but $K_2$ cannot.  By
\cite{NOT}, $\hthe (K_2) \neq 0$, thus by Proposition \ref{vanish}, $K_2$ is
right-veering. Since $sl(K_2)=sl(K_1)$ is non-maximal, the inequalities
$sl< s-1$ and $sl <2 \tau-1$ must be strict. The knots $K_1$, $K_2$
can be represented as transverse push-offs of Legendrian knots $L_1$, $L_2$
such that $L_1$ and $L_2$ differ by a ``SZ-move'' \cite[Figure 6]{NOT}; 
by \cite[Proposition 2.8]{LNS} this implies that $K_1$ and $K_2$ can be related
via ``negative flypes''
(see \cite{LNS} for definitions of these modifications), and then by
\cite[Theorem 1.2]{HKP}
it follows that the branched double covers $(\Sigma(K_1), \xi_{K_1})$ and 
$(\Sigma(K_2), \xi_{K_2})$ are contactomorphic. Since 
$(\Sigma(K_1), \xi_{K_1})$ is overtwisted, so is $(\Sigma(K_2), \xi_{K_2})$.
Finally, $\psi(K_2) =0$ (as stated in \cite{NOT}, the Khovanov homology
component of the relevant bi-degree is trivial).
\end{proof}

Note, however, that the converses of Propositions \ref{vanish},
\ref{sl}, and \ref{branch} hold if the braid is {\em quasipositive}
(rather than simply right-veering). Recall 

\begin{definition} A braid $\sigma$ is called quasipositive if it is
a product of conjugates of standard generators, i.e. 
$\sigma = \prod w \sigma_i w^{-1}$. 
\end{definition}

As an automorphism of the punctured disk, $\sigma$ is represented as a product
of (non-standard) positive half-twists.

\begin{definition} We say that a transverse link $L$ is quasipositive if $L$
can be represented by a quasipositive braid. 
\end{definition}

 Orevkov proved \cite{Or} that a quasipositive
braid remains quasipositive after a positive Markov destabilization. Thus, in
contrast to Proposition \ref{prop-stab-rv}  {\em every}
braid representing a quasipositive transverse link will be quasipositive.   

The following properties are well-known. The statements about transverse
invariants follow from Proposition \ref{functo}, as one can resolve a number of
positive crossings in a quasipositive braid to obtain the trivial braid.

\begin{prop} \label{quasipos} Suppose $L$ is a quasipositive transverse knot.
Then

1) $sl(L) = 2 g^*-1 = 2 \tau -1 = s -1$; 

2) branched double cover  $(\Sigma_(L),\xi_L)$ is Stein fillable, 

3) transverse invariants $\psi$ and $\hthe$ do not vanish. 

\end{prop}

In a sense, quasipositivity is a very strong condition. It would be
interesting
to define a weaker notion of ``strong right-veering''  that 
would still imply the analog of Proposition \ref{quasipos}.

\section{Right-veering 3-braids}

In this section, we prove the following 

\begin{theorem} \label{theta-3braids} Let $K$ be a transverse link represented
by a 3-braid $\beta \in B_3$. Then $\beta$ is right-veering if and only if
$\hthe(K) \neq 0$.
\end{theorem}

\begin{proof} If $\beta$ is non-right-veering, then $\hthe(K)= 0$ by
Proposition \ref{vanish}. We need to prove that $\hthe(K) \neq 0$ for
right-veering braids.
 
According to Murasugi's classification \cite{Mur}, 3-braids come in the
following types (up to conjugation):

(a) $h^d \sigma_1 \sigma_2^{-a_1} \sigma_1 \sigma_2^{-a_2} \dots
\sigma_1 \sigma_2^{-a_n}$, where $a_i \geq 0$, some $a_i>0$;

(b) $h^d \sigma_2^m$, where $m \in \ZZ$;

(c) $h^d \sigma_1^m \sigma_2 ^{-1}$, where $m=-1, -2, -3$.  
 
Here, $h=(\sigma_1 \sigma_2)^3$ is a full positive twist of the 3-braid, and
the exponent $d$ is an arbitrary integer.  
 
We notice that some of the cases above clearly give non-right-veering braids.
Indeed, if $d\leq 0$  in (a), (c), or $d<0$ in (b), or $d=0$, $m<0$ in (b), we
get braid words where $\sigma_2$ or $\sigma_1$ has only negative exponents. To
prove the theorem, 
it suffices to show that $\hthe(K)$ is non-zero in all of the remaining cases, 
namely

(a') $h^d \sigma_1 \sigma_2^{-a_1} \sigma_1 \sigma_2^{-a_2} \dots
\sigma_1 \sigma_2^{-a_n}$, where $d > 0$, $a_i \geq 0$, some $a_i>0$;

(b') $h^d \sigma_2^m$, where $d >0$,  $m \in \ZZ$, or $d=0$, $m\geq 0$; 

(c') $h^d \sigma_1^m \sigma_2 ^{-1}$, where $d>0$, $m=-1, -2, -3$.

Next, all braids in (c') are clearly quasipositive, so $\hthe(K)$ is non-zero
by \cite{Baldwin}; see also Proposition \ref{quasipos} above. Non-vanishing is
also obvious for the
braids in (b') with $m\geq 0$, since these are all positive.

Next, we will show that
\begin{equation} 
\label{eq:model}
\hthe(h \cdot \sigma_2 ^{-k}) \neq 0\text{ for all }  k>0.
\end{equation}

These are braids from (b') with $d=1$. Once this case is established, the rest
of (b') and (a') follows from the functoriality (Proposition \ref{functo}):
indeed, all the
other braids in (b') and (a') can be obtained from the model braid  $h \cdot
\sigma_2 ^{-k}$ by insertion of additional positive crossings. 

Observe that 
$ h \cdot \sigma_2^{-k} = (\sigma_1 \sigma_2)^3 \sigma_2^{-k} = 
\sigma_1
\sigma_2 \sigma_2 \sigma_1 \sigma_2^{-k+2}$.

\begin{figure}[ht]
  \includegraphics[scale=0.8, keepaspectratio]{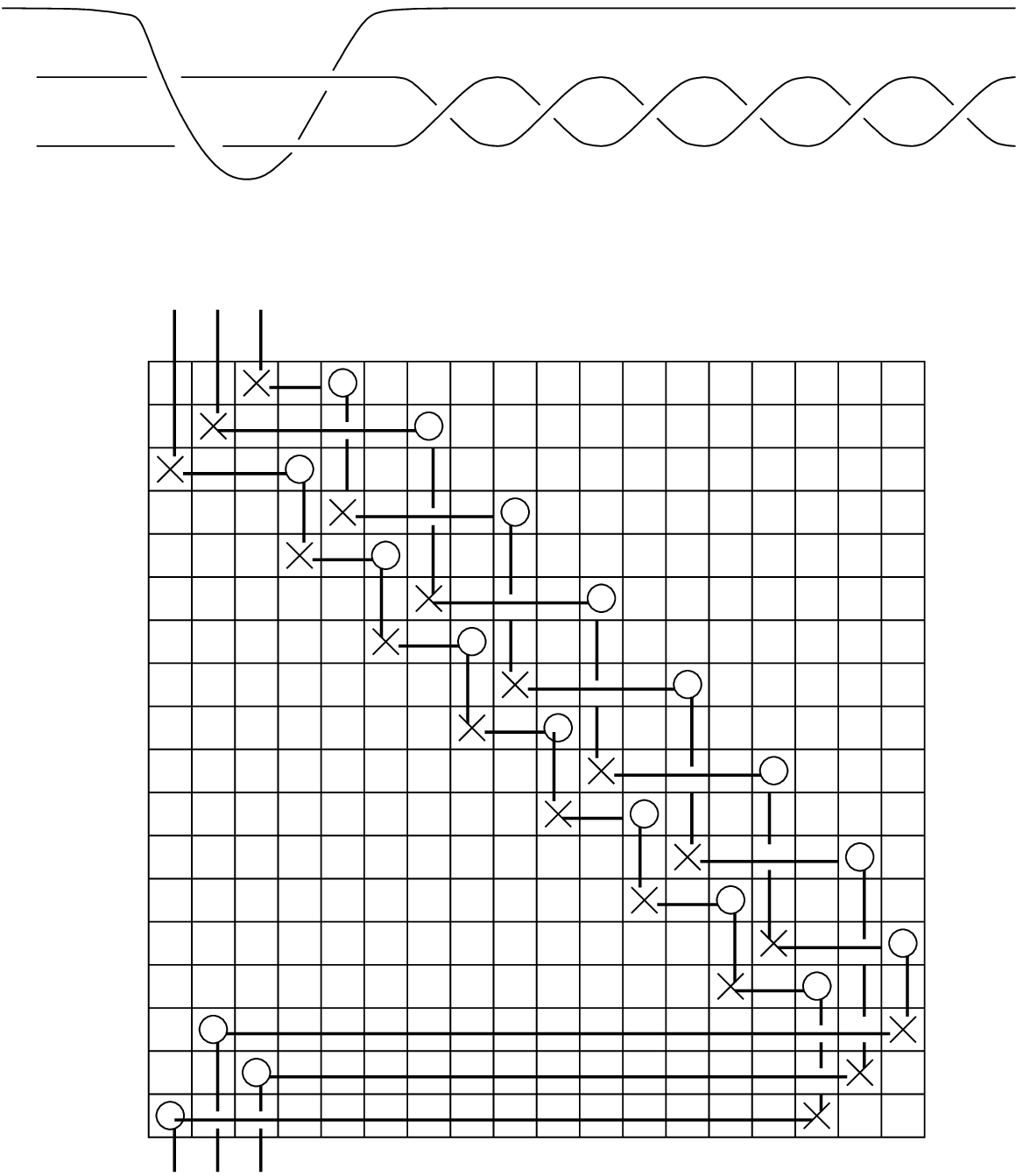}
  \caption{A grid diagram for $\sigma_1 \sigma_2 \sigma_2 \sigma_1
\sigma_2^{-k}$.}
  \label{grid1-3braid}
\end{figure}

We establish (\ref{eq:model}) for braids of the form 
$\sigma_1 \sigma_2 \sigma_2 \sigma_1 \sigma_2^{-k}$ with $k>0$ 
by a direct
examination of a grid diagram representing the corresponding transverse link.
The grid diagram for 
is shown
in Figure \ref{grid1-3braid}; in our picture $k=6$, but the general pattern
should be clear to the reader. The element $\hthe \in \HFK (m(K))$ is given by 
the cycle made of ``upper right corners'' of the $X$-markings in the figure. 
If this cycle is null-homologous, it is the boundary of another element in 
$\CFK(m(K))$. Now, recall the definition of the boundary map in grid homology
\cite{MOS}. For a generator $\mathbf{x}$ of $\CFK$, given by an $n$-tuple of
intersection points $\mathbf{x}=(x_1,x_2, \dots, x_n)$, the boundary
$\partial \mathbf{x}$ is obtained by summing over all $n$-tuples of the form 
$\mathbf{y}= (x_1, \dots, y_i, \dots, y_j, \dots x_n)$ that differ from 
$\mathbf{x}=(x_1,x_2, \dots, x_n)$ by exactly two entries, and such that there
is an empty rectangle in the grid diagram whose top right and bottom left
corners are $x_i, x_j$, while top left and bottom right corners are $y_i, y_j$.
(A rectangle is ``empty'' if it contains no X and O markings and none of the
entries of  $\mathbf{x}$ and $\mathbf{y}$.) 
Examining the diagram, we immediately see that no such empty rectangles exist
in Figure
\ref{grid1-3braid}. 
\end{proof}

\begin{figure}[ht]
  \includegraphics[scale=0.8, keepaspectratio]{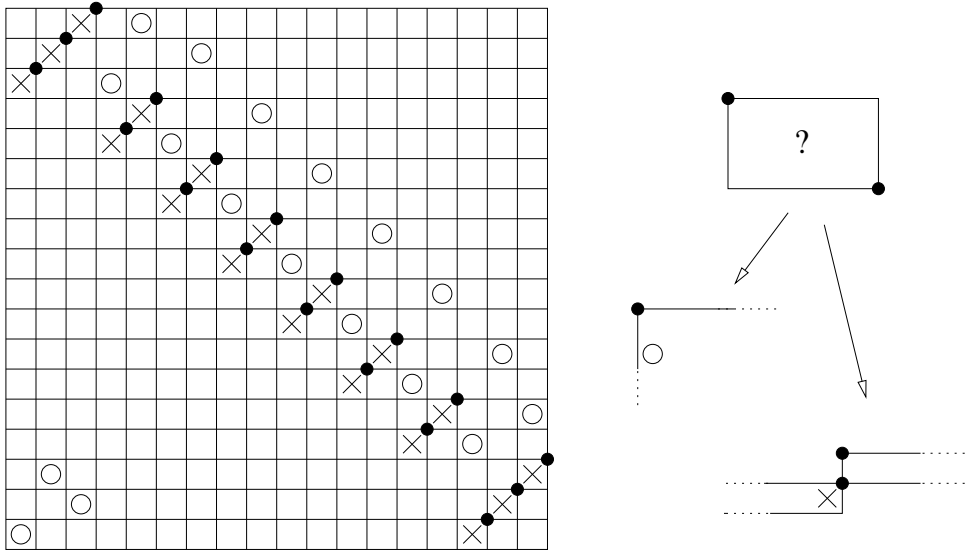}
  \caption{The cycle $\hthe$. Looking for empty triangles.}
  \label{grid-3braid}
\end{figure}

It should be noted that 3-braids are quite special and satisfy many properties
that might not hold in general. In particular, we have the following

\begin{prop} \label {3braid-covers} A transverse 3-braid is right-veering if and
only if its branched double cover is tight.
\end{prop}

\begin{proof} The ``if'' part is always true by Proposition \ref{branch}; 
we now establish the ``only if'' part, which is specific to 3-braids.
This is an easy corollary of known results. Indeed, the branched
double cover of a 3-braid has an open book decomposition whose page is a
punctured torus (or rather, a genus one surface $S$ with one boundary
component).
The page arises as a double cover of the disk branched over 3 points, and the
braid monodromy lifts to a diffeomorphism $\phi$ of the page $S$, i.e.
the monodromy of the open book. Every arc $\alpha$ with endpoints on $\d S$
covers an arc $a$ in $(\D, \{x_1, x_2, x_3\})$. (Note
that this statement is specific to the genus 1 case.) Indeed,
$\alpha$ can be represented as an image $\rho(\alpha_0)$ of a standard arc
$\alpha_0$ in $S$ covering an arc $a_0$ in  $(\D, \{x_1, x_2, x_3\})$, under
some diffeomorphism $\rho \in Map(S)$. Since $S$ has genus one and one boundary
component, any such $\rho$ is a lift of an element  
$r \in Map (\D, \{x_1, x_2, x_3\})$, so that $\alpha$ covers $r(a_0)$. 

Now suppose that the given transverse 3-braid is right-veering. Then the
corresponding open book for the double cover must be right-veering as well, 
since a non-right-veering arc ``upstairs'' would cover a non-right-veering arc
``downstairs''. Indeed,
if the image $\phi(\alpha)$ is to the left of $\alpha$ for some arc $\alpha$
in $S$, 
and $\alpha$ covers an arc $a$ in $(\D, \{x_1, x_2, x_3\})$, then the image of
$a$ under the braid monodromy must lie to the left of $a$ (one should be
careful with non-essential intersections, but the arcs can be pulled taut both
upstairs and downstairs as in the proof of Proposition \ref{branch}).
Once we know that $(S, \phi)$ is right-veering, we can  use a result of
Honda--Kazez--Mati\'c \cite{HKM-cont}: right-veering open books with a punctured
torus page support tight contact structures. 
\end{proof}

Combining Proposition \ref{3braid-covers}, Proposition \ref{vanish},
and
 \cite[Theorem 1.2]{HKM-cont}, we have 

\begin{cor} Let $K$ be a transverse 3-braid, and suppose that $\hthe(K)
\neq 0$ or $\psi(K) \neq 0$. Then the branched double cover $(\Sigma(K),
\xi_K)$ is tight; moreover, the Heegaard Floer contact invariant $c(\xi_K)$ is
non-zero. 
\end{cor}

It is interesting to compare the above corollary with the results of
\cite{BaPl}, where in certain cases non-vanishing of $c(\xi_K)$ is derived from
non-vanishing of $\psi(K)$ via the Ozsv\'ath--Szab\'o spectral sequence (see
\cite{OS-double}, \cite{LRob}).

The analog of Theorem \ref{theta-3braids} does not hold for the transverse
invariant $\psi$ in Khovanov homology. Indeed, while right-veering braids from
(a') are all quasi-alternating and have $\psi \neq 0$, the  braid $h \cdot
\sigma_2 ^{-k}$ from (b') considered in the proof above has $\psi(h \cdot
\sigma_2^{-k}) = 0$. See \cite{BaPl} for discussion.

\section{Fractional Dehn twist coefficient} \label{FDTC}

It is natural to ask whether Theorem \ref{theta-3braids} extends to
braids of higher index. It would perhaps be too optimistic to expect that 
the transverse invariant $\hthe$ is non-zero for {\em all} right-veering
braids. Our Theorem \ref{C>1} establishes that $\hthe$ does not vanish as long
as the braid has enough positive twisting.


We will obtain Theorem \ref{C>1} as a corollary of a stronger
statement, Theorem \ref{D>1} below. Theorem \ref{D>1} uses braid orderings and 
is stated in terms of Dehornoy's floor of a braid. We first recall the notion of
Dehornoy floor and prove Theorem \ref{D>1}. Then, we will discuss 
the fractional Dehn twist coefficient and its relation to Dehornoy's floor and
complete the the proof of Theorem \ref{C>1}.

As usual, let $\Delta \in B_n$ denote the
Garside element, 
$$\Delta = (\sigma_1 \sigma_2 \dots \sigma_{n-1})
(\sigma_1 \sigma_2 \dots \sigma_{n-2})\dots (\sigma_1 \sigma_2) \sigma_1.
$$
Recall that $\Delta$ generates the center of the group $B_n$.
Geometrically, $\Delta$ is the positive half-twist on all strands, so 
$$
\Delta ^2 = (\sigma_1 \sigma_2 \dots \sigma_{n-1})^n
$$ is the full twist.

\begin{definition} \cite{De} Let $\beta \in B_n$ be an arbitrary braid. Let $m$
be an
integer such that 
$\Delta^{2m} \preceq \beta \prec \Delta^{2m+2}$, where $\prec$ is Dehornoy's
ordering on $B_n$.
Then $m= \lfloor \beta \rfloor_D$ is called Dehornoy's floor of $\beta$.
\end{definition}

Note that several slightly
different definitions of $\lfloor \beta \rfloor_D$ exist in the literature
\cite{MN,It1,De}; the present one best fits our purposes.
Intuitively, Dehornoy's floor tells us how many positive full twists can be
extracted from a given braid, so that the ``leftover'' braid is still
non-negative in Dehornoy's ordering. (Since $\prec$ is a well-ordering, 
and for any given braid $\beta$ we have that $\Delta^{-N} \prec \beta \prec
\Delta^N$ when $N$ is very large, the floor function is well-defined.)
Note that Dehornoy's floor is
{\em
not} conjugacy invariant (and thus not an invariant of {\em closed braids}): for
example, $\lfloor \Delta^2 \sigma_1 \sigma_2^{-1} \rfloor_D =1$ but 
$\lfloor \Delta^2 \sigma_2 \sigma_1^{-1}  \rfloor_D =0$, despite the fact
that these two 3-braids are conjugate.

We are now ready to state our result.

\begin{theorem}   \label{D>1} If the transverse link $K$ can be represented by 
a braid $\beta$ with $\lfloor \beta \rfloor_D \geq 1$, then $\hthe(K) \neq 0$.
\end{theorem}

\begin{figure}[htb]
  \includegraphics[scale=0.8, keepaspectratio]{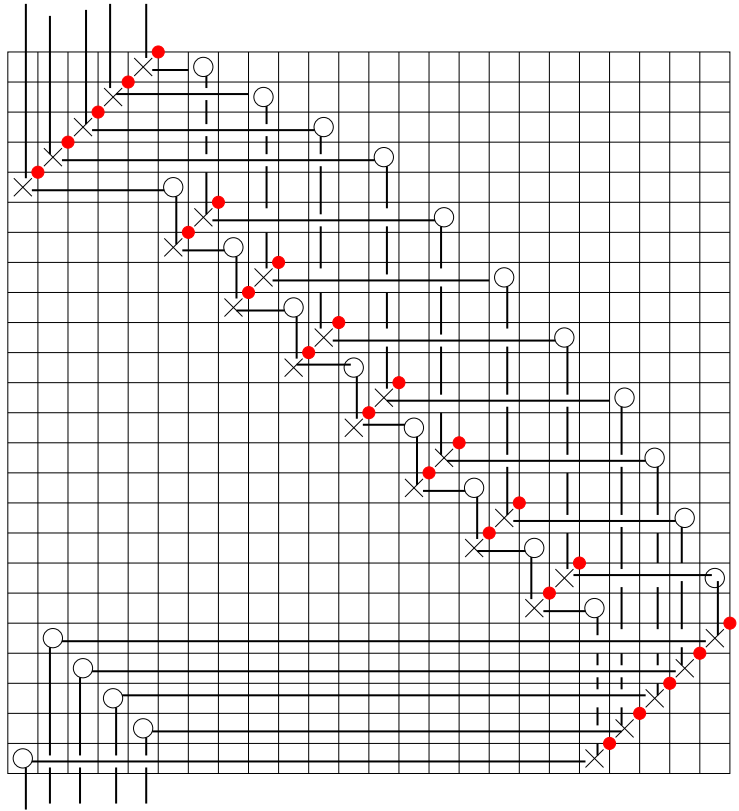}
  \caption{The cycle $\hthe$ for a braid $\beta_{n,k}$.}
  \label{big-grid}
\end{figure}

\begin{proof} We will use the same idea as in Theorem \ref{theta-3braids}: the
question can be reduced to a family of model braids, which can then be attacked
directly by via grid diagrams. 

Indeed, suppose that $\lfloor \beta \rfloor_D \geq 1$. This means that $\beta =
\Delta^2 \gamma$, where  $\gamma \succeq 1$ with respect to Dehornoy's order.
Then, $\gamma$ has a braid word which is either $\sigma_1$-positive or
$\sigma_1$-free. Both cases guarantee that the braid word for $\gamma$ has no 
entries of $\sigma_1$ with negative exponent.

For any $k>0$, consider a braid  $\beta_{k,n} \in B_n$ of the form 
$$
\beta_{k}= (\sigma_1 \sigma_2 \dots \sigma_{n-1}) 
(\sigma_{n-1} \dots \sigma_2 \sigma_1) (\sigma_2 \sigma_3 \dots
\sigma_{n-1})^{-k}.
$$
A braid $\beta_{n,k}$ is shown in Figure \ref{big-grid}. In these braids, all
the strands
except the first are  twisted together a
number of times in the negative direction, while the first strand twists
positively once around all the other strands. 

The family of braids $\beta_{k,n}$ is universal in the following sense: 
any given braid $\beta \in B_n$ with $\lfloor \beta \rfloor_D \geq 1$
can be obtained from a braid $\beta_{k,n}$ with $k$ large enough by insertion 
of a number of positive crossings. (These extra positive crossings will cancel
some of the negative crossings in $\beta_{k,n}$ and create the positive
crossings present in $\beta$; it is essential to notice that $\beta$ has no
negative crossings in the $\sigma_1$ level.) By functoriality of $\hthe$
\cite{Baldwin} (see  Lemma \ref{functo} above), now it suffices to check that
$\hthe(\beta_{k,n}) \neq 0$ for all $k,n$.

We examine a grid diagram for a braid  $\beta_{k,n}$, shown in Figure
\ref{big-grid}. This diagram has a similar pattern and the same features
as the diagram for a 3-braid we considered in  
Figure \ref{grid1-3braid}. As before, we can see that there are no empty
rectangles 
whose top left and bottom right corners would be given by the distinguished
intersection points forming the $\hthe$-cycle. Therefore, $\hthe \neq 0$.
\end{proof}

We now discuss the definition of the  fractional Dehn twist coefficient; 
its relation to Dehornoy's floor will give an immediate proof of Theorem
\ref{C>1}. The idea of the fractional Dehn twist coefficient 
first appeared in \cite{GO} and was developed in the context of open books and
contact topology in \cite{HKM}. For classical braids, 
a similar notion (via a somewhat different approach) was studied in \cite{Ma}.
A generalization of FDTC to the case of braids in arbitary open books, and a
detailed proof that different definitions are equivalent, is given in \cite{IK}.

To define FDTC of a given braid $\beta \in B_n$, equip the disk $\D \setminus
\{x_1, x_2, \dots x_n\}$ with a complete hyperbolic metric of finite volume
such that $\d \D$ is geodesic. Let $\ti{\D}$  be the universal cover of the
punctured disk;  $\ti{\D}$ admits an isometric embedding into the Poincar\'e
disk $\HH^2$. We can compactify  $\ti{\D}$ by adding points at infinity; the
result is a closed disk $\bar{\D}$. Choose a basepoint $* \in \d \D\setminus
\{x_1, x_2, \dots, x_n\}$ and its lift $\ti{*} \in \d \bar{\D}$. Let $\beta \in
B_n = Map (\D, \{x_1, \dots x_2\})$ be a braid. Consider the lift
$\ti{\beta}: \ti{\D} \to \ti{\D}$ with $\ti{\beta}(*) = \ti{*}$. It extends
uniquely to a homeomorphism $\bar{\beta}: \bar{\D} \to \bar{\D}$ that fixes  
the component $\ti C$ of the preimage $\pi^{-1}(\d \D)$ of $\d \D$ under the
covering map $\pi$, such that $\ti C$ contains $\ti{*}$. Notice that $\d {\bar
\D} \setminus \ti C$ can be identified with $\RR$, so $\bar{\beta}$ induces a
homeomorphism of $\RR$. Moreover, the identification  $\d {\bar
\D} \setminus \ti C \equiv \RR$ can be chosen so that the full twist around $\d
\D$ (i.e. the braid $\Delta^2 \in B_n$) induces the homeomorphism  
$x \mapsto x+1$. Since $\Delta$ is in the center of $B_n$, any braid $\beta
\in B_n$ will induce a homeomorphism of $\RR$ that commutes with the
translation $x \mapsto x+1$. This means that we have a map 
$$
\Theta: B_n \to \widetilde{Homeo}^+ (S^1),
$$
where ${Homeo}^+ (S^1)$ is the group of homeomorphisms
of $\RR$ that are lifts of orientation-preserving homeomorphisms of $S^1$.
The map $\Theta$ is called the {\em Nielsen--Thurston} map. 

Recall that a basic invariant of a map  $ h \in {Homeo}^+ (S^1)$ is given by its 
{\em translation number} defined as 
$$
T(h) = \lim_{N \to \infty} \frac{h^N(x) -x}{N},
$$
for any $x \in \RR$. Finally, 

\begin{definition} The fractional Dehn twist coefficient of a braid $\beta \in
B_n$ is defined as the translation number of $\Theta(\beta)$:
$$
C = T(\Theta(\beta)).
$$ 
\end{definition}
The FDTC is well-defined, ie independent of the choice of the hyperbolic metric
on the punctured disk and other choices, \cite{Ma}. 

Another definition, emphasizing the meaning of FDTC as the
amount of rotation about the boundary of $\D$, can be given in the spirit of
\cite{HKM}. One uses the Nielsen-Thurston
classification to find a free isotopy connecting (an iterate of)  $\phi$ to its
pseudo-Anosov, periodic, or reducible representative, and considers 
the winding number of an arc traced by a basepoint in $\d \D$ under this
isotopy. We refer the reader to \cite{IK} for a detailed definition and
discussion.

The fractional Dehn twist coefficient is related to  Dehornoy's floor as
follows. (Note that unlike Dehornoy's floor, FDTC is invariant under
conjugation and thus gives an invariant of a {\em closed} braid.)

\begin{lemma} \cite{Ma} \label{CvsD}  Let  $\beta \in B_m$ be a braid with
fractional Dehn twist coefficient $C$. Then 
$$
  \lfloor \beta \rfloor_{D} +1  \geq  C \geq \lfloor \beta \rfloor_{D} \quad
\text{ and  } \quad C  = \lim_{n \to \infty} \frac{\lfloor \beta^n
\rfloor_{D}}{n}.  
$$
\flushright \qed
\end{lemma}

\begin{proof}[Proof of Theorem \ref{C>1}] The result follows immediately from 
Theorem \ref{D>1} and Lemma \ref{CvsD}, since any  $\beta \in B_m$ representing
a  closed braid with fractional Dehn twist coefficient $C>1$
has $\lfloor \beta \rfloor_{D} \geq 1$. (If $\lfloor \beta \rfloor_{D} \leq 0$,
then $C \leq 1$.)
\end{proof}

\section{Higher-order simple covers}

We have seen that the properties of contact manifolds arising as double covers
of $S^3, \xi_{std}$ branched over transverse braids often reflect the
properties of the branch locus. Indeed, quasipositive braids give rise to Stein
fillable covers, while non-right-veering braids have overtwisted covers. 
In this section we turn attention to higher-order covers. If one considers 
cyclic covers, the same properties hold: 

\begin{theorem} Let $(\Sigma_n(K), \xi_n)$ be the $n$-fold cover of $(S^3,
\xi_{std})$ branched over a transverse link $K$. Then 

1) If $K$ is represented by a quasipositive braid, $(\Sigma_n(K), \xi_n)$ is
Stein fillable.

2) If $K$ is a non-right-veering transverse link, $(\Sigma_n(K), \xi_n)$ is
overtwisted.
\end{theorem}

\begin{proof} The statements above are essentially contained in
\cite[Theorem 1.3, Proposition 4.2]{HKP};
while Proposition 4.2 in \cite{HKP} is stated in terms of braids ``with a row of
negative crossings'', by Corollary \ref{non-rv-form} this is equivalent to
non-right-veering.
\end{proof}

The situation changes dramatically if we consider {\em simple} covers instead
of {\em cyclic} ones. It is known that any contact 3-manifold can be
represented as a 3-fold cover of $(S^3, \xi_{std})$ branched over some
transverse link $K$; thus 3-fold simple covers are of particular interest. 
We show that in this setting, positivity and right-veering properties of the
branch locus have no bearing on the properties of the cover:

\begin{theorem} Any contact 3-manifold can be represented both as a 3-fold
simple cover of a positive braid and a 3-fold
simple cover of a negative braid.
\end{theorem}

\begin{proof} If we ignore the contact structures, the counterpart of this
statement is known in low-dimensional topology: any 3-manifold can be
represented as a 3-fold simple cover of a positive braid (or of a negative
braid). The proof (see eg \cite[Theorem 25.2, Step 2]{PS}) is based on the fact
that the branch locus can be modified in certain ways without affecting the
3-fold simple branched cover. One of the transformations \cite[Figure 24.3]{PS}
replaces two parallel strands in the braid by three positive crossings,
provided that the branching pattern is as shown in Figure \ref{branch-locus}.
(As usual, the numbers indicate how the sheets of the cover meet the branch
locus; for example, one passes from sheet 1 to sheet 2 after going along a
small loop around a given strand.) We now  observe that the same proof
goes through in presence of contact structures, because the branched double 
cover corresponding to each of the transverse tangles in Figure
\ref{branch-locus} is a Darboux ball. (Namely, we close up these tangles to
obtain the trivial braid on 2 strands and the braid $\sigma_1^3$; both
3-fold simple branched covers are the standard contact 3-spheres. Indeed, 
following the discussion of  \cite[Figure 23.7]{PS}, one sees that the cover
corresponding to the braid $\sigma_1^3$ in Figure \ref{branch-locus} is given
by an open book with a disk page. This discussion is compatible with contact
structures, so the cover is $(S^3, \xi_{std})$.) 

\begin{figure}[htb]
\bigskip
	\labellist
	\small\hair 2pt
	\pinlabel $12$ at 130 30
	\pinlabel $13$ at 130 100
	\pinlabel $12$ at 200 30
	\pinlabel $13$ at 200 100
	\pinlabel $12$ at 440 30
	\pinlabel $13$ at 440 100
	\pinlabel $23$ at 275 100
	\endlabellist
\centering	
  \includegraphics[scale=0.7]{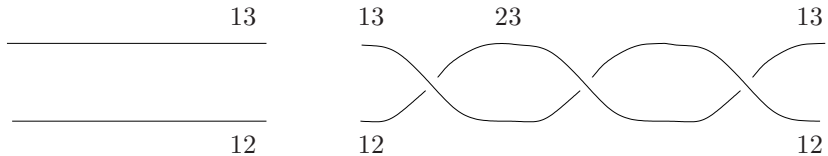}
  
  \caption{Modification of the branch locus for a 3-fold simple cover.}
  \label{branch-locus}
\end{figure}

Inserting triples of positive crossings into a given transverse braid, we can
obviously convert it into a positive braid. The branched double cover will
remain the same contact manifold during this procedure, so we have shown that
any contact manifold is a 3-fold simple cover of a positive braid. Performing
the same modification backwards, we can insert a triple of negative crossings
into any branch locus (without changing the cover); this shows that any contact
3-manifold is a cover of $(S^3, \xi_{std})$ branched over a negative braid. 
\end{proof}

\end{document}